\documentclass{article}

\def\Anonimo{1}  

\usepackage{fullpage}
\usepackage[utf8]{inputenc}
\usepackage{amsmath, amsfonts,amsthm, bm}
\usepackage{mathtools}
\usepackage{thm-restate}
\usepackage[dvipsnames]{xcolor}
\usepackage{caption}
\usepackage{hyperref}
\hypersetup{
    colorlinks=true,       
    linkcolor=blue,        
    citecolor=red,         
    filecolor=magenta,     
    urlcolor=cyan,         
    linktocpage=true
}
\usepackage[capitalise,noabbrev]{cleveref}

\usepackage{tikz}
\usetikzlibrary{arrows,arrows.meta, positioning,calc,shapes.geometric, math}
\definecolor{orangeDark}{RGB}{204,76,2}
\definecolor{orangeMedium}{RGB}{254,153,41}
\definecolor{orangeLight}{RGB}{254,217,142}
\definecolor{blueishLight}{RGB}{189,215,231}
\definecolor{blueishMedium}{RGB}{107,174,214}
\definecolor{blueishDark}{RGB}{33,113,181}
\definecolor{pinkLight}{RGB}{251,180,185}
\definecolor{pinkMedium}{RGB}{247,104,161}
\definecolor{pinkDark}{RGB}{197,27,138}
\definecolor{pinkDark2}{RGB}{122,1,119}
\definecolor{greenMedium}{RGB}{173,221,142}
\definecolor{greenDark}{RGB}{49,163,84}

\usepackage[skip=6pt]{parskip}
\usepackage{setspace}
\usepackage{etoolbox}
\AtBeginEnvironment{thebibliography}{\linespread{1}\selectfont}

\theoremstyle{plain}
\newtheorem{thm}{Theorem}
\newtheorem{lem}[thm]{Lemma}

\newtheorem{cor}[thm]{Corollary}

\newtheorem{conjec}[thm]{Conjecture}

\theoremstyle{definition}

\newcommand{\p}[1]{p_{(#1)}}
\usepackage{authblk}
\title{On Mixed Cages of Girth 6}
\author[1]{Gabriela Araujo-Pardo}
\author[2]{Mirabel Mendoza-Cadena}
\affil[1]{Instituto de Matemáticas, Universidad Nacional Autónoma de México, Campus Juriquilla, Querétaro, Mexico.}
\affil[2]{Centro de Modelamiento Matemático (CNRS IRL2807), Universidad de Chile, Santiago, Chile.\\
Emails: \texttt{garaujo@im.unam.mx, lmmendoza@cmm.uchile.cl}}
\date{ }

\begin{document}
\ifnum\Anonimo=1 {
    \maketitle}
\else{\begin{center}
    \Large On Mixed Cages of Girth 6
\end{center}}
\fi

\begin{abstract}
    A $[z,r;g]$-mixed cage is a mixed graph of minimum order such that each vertex has $z$ in-arcs, $z$ out-arcs, $r$ edges, and it has girth $g$. 
    We present an infinite family of mixed graphs with girth 6. This construction also provides an upper bound on the minimum order of mixed cages of girth 6. Additionally, we introduce a lower bound on the minimum order for any mixed cage. \medskip
    
\end{abstract}

\section{Introduction \label{sec:intro}}
Mixed regular cages where introduced by Araujo-Pardo,  Hernández-Cruz, and Montellano-Ballesteros~\cite{araujo2019mixed}.

A \emph{mixed graph} is a simple\footnote{No parallel arcs or edges are allowed, neither a parallel arc and edge.} graph $G=(V;E\cup A)$ where $V(G)$ is the set of vertices, $E(G)$ is the set of edges and $A(G)$ the set of arcs. We denote an edge between $u$ and $v$ by $uv$ and an arc from $u$ to $v$ by $(u,v)$. We say that $u$ and $v$ are \emph{edge-adjacent} (\emph{arc-adjacent}) if they are connected by an edge (arc).

Walks, paths and cycles are referred  by a sequence of vertices $(v_0, v_1, \dots, v_n)$, where $v_i$ and $v_{i+1}$ are connected by either the arc $(v_i,v_{i+1})$ or by the edge $v_iv_{i+1}$. The \emph{girth} of $G$ is the length of the shortest cycle.

If each vertex has  $z$ in-arcs, $z$ out-arcs, and $r$ edges, then the graph is called \emph{mixed regular graph}. If, additionally, it has girth $g$, then we say that $G$ is a \emph{$[z,r;g]$-mixed graph}.

A \emph{$[z,r;g]$-mixed cage} is a $[z,r;g]$-mixed graph of minimum order, where the \emph{order} of a graph is the number of vertices. The minimum order of $[z,r;g]$-mixed graphs is denoted by $n[z,r;g]$.

\paragraph{Previous work.} As mentioned above, in  \cite{araujo2019mixed} it is described the notion of mixed regular graphs for the first time. Explicit mixed cages for given values and some lower bounds were presented. Additionally, they gave a lower bound on the order of $[1, r; g]$-mixed cages, also known as the AHM bound. This bound relies on the Moore’s lower bound (see e.g.~\cite{exoo2012dynamic}). 
The authors~\cite{araujo2019mixed} also showed that the AHM bound is achieved for $r=2$ and any $g$. Exoo~\cite{exoo2023mixed} showed that the AHM bound can be achieved for $r=3$ and $g=5,6$.

Later, Araujo-Pardo, {De la Cruz}, and González-Moreno~\cite{araujo2022monotonicity} showed that the order of a $[z,r;g]$-mixed cage is a monotone function with respect to $g$ if $z\in \{ 1,2 \}$. Moreover, they proved that such $[z,r;g]$-mixed cages are 2-connected, and for general $z$ the graphs are strongly connected. Further constructions and lower bounds were presented for $g = 5$ too. This later uses incidence structures, in particular, the elliptic semiplane of type $C$. Next section includes more details on this topic.

Exoo~\cite{exoo2023mixed} focused on $z \in  \{1, 2\}$ and presented a series of mixed cages and upper bounds for different values of $r \leq 5$ and $g \leq 8$, by the use of computer searches.
Jajcayová and Jajcay~\cite{jajcajova2024totallyregular} developed $[z,r;g]$-mixed graphs by replacing some edges of specific regular undirected graphs with arcs.
Recently, Araujo-Pardo and Mendoza-Cadena \cite{araujopardo2025noteGirthFive} showed a construction for mixed cages of girth 5 which gives a better bound for the order of mixed cages than previous works.

\paragraph{Our results.} Section~\ref{sec:preliminaries} provides the necessary background for our constructions. In \cref{sec:family_girth_6}, we define an infinite family of mixed graphs with a girth of 6, utilizing the incidence graph of the biaffine plane over the field $\mathbb{Z}_q$, where $q$ is a prime number. This construction establishes upper bounds for the minimum order of mixed graphs with a girth of 6. In \cref{sec:lower_bounds}, we present a lower bound for the minimum order of $[z,r;g]$-mixed cages for any given values of $z$, $r$, and $g$. Finally, in \cref{sec:1-3-6_construction}, we introduce a construction of the $[1,3;6]$-mixed cage derived from the projective plane over the Galois field of order 4.

\section{Preliminaries \label{sec:preliminaries}}
To be self-contained, we describe in detail three graphs that are obtained from projective planes of order $q$. 
For an introduction see e.g. \cite{balbuena2008incidence, van2001course, kiss2019finite, godsil2013algebraic}.
Constructions using the incidence graph of a projective plane or their substructures have been widely used. 
The first construction was presented by Miller and Sir{\'a}n \cite{miller2012moore}, followed by Hafner~\cite{hafner2004geometric},  Araujo-Pardo,  Noy and Serra~\cite{araujo2006geometric}, and
Araujo-Pardo, Balbuena, Miller and Ždímalová~\cite{araujo2017family}, just to mention some examples.
Regarding mixed cages, Araujo-Pardo, {De la Cruz}, and González-Moreno~\cite{araujo2022monotonicity} generated a family of $[z, r;5]$-mixed graphs using this technique. 

\paragraph{The graph $\bm{G_{(2,q)}}$ of the projective plane on the Galois field.}  Let $q$ be a prime power, and let $\mathbb{G}_q$ be the Galois field of order $q$. 
We describe the projective plane $PG(2; q)$.
Let $L_\infty$ and $P_\infty$ be the incident line and point in the projective plane, called the infinity line and infinity point, respectively.
The classes of lines (points)  are denoted by $L_i$ ($P_i$)  for $i \in \mathbb{G}_q$.
Each line $L$ have $q+1$ points and each point $P$ is incident to $q+1$ lines. 
Any line in $L_m$ described by $y = m x +b$  is denoted by $[m,b]$, and any point in $P_x$ is denoted by $(x,y)$. 
For each $i \in \mathbb{G}_q$, the set of points incident to $L_i$ is $\{ (i,j) \, \, \vert \, \, j \in  \mathbb{G}_q\}$ and the set of lines incident to $P_i$ is $\{ [i,j]\, \vert \, j \in  \mathbb{G}_q \}$.
Finally, for $m,b \in \mathbb{G}_q$ the line $[m,b]$ is incident to all points $(x,y)$ such that $y = m x +b$ holds for $x,y \in \mathbb{G}_q$.

The projective plane $PG(2; q)$ has an incidence graph associated that we denote by $G_{(2,q)}$. Each line and point has a unique associated vertex $v \in V(G_{(2,q)}) $. 
Similarly, $uv \in E(G_{(2,q)})$ if and only if  their associated line and point are incident in  $PG(2; q)$. For an easy reading, we may refer to ``line $[m,b]$'' (point $(x,y)$) instead of ``vertex associated to the line $[m,b]$'' (point $(x,y)$). Note that $G_{(2,q)}$ has order $2q^2 + 2q + 2$, diameter 3 and girth 6. In fact, it is known that this incidence graph is a $[0,q;6]$-mixed cage (undirected cage) that attains the Moore bound.

\paragraph{The bipartite graph $\bm{B_q}$.} The corresponding plane of the following incidence graph is also known as the biaffine plane. Consider the graph $G_{(2,q)}$ constructed above and delete the vertices $L_\infty, P_\infty,$ and  $ L_i,P_i$ for $i \in \mathbb{G}_q$. We are left with a bipartite graph that we denote by $B_q$.  Note that $B_q$ consists only on points $(x,y)$ and lines $[m,b]$, for $x,y,m,b \in \mathbb{G}_q$, and it has order $2q^2$, diameter 4 and girth 6 if $q\geq 3$, and girth 8 if $q=2$.
 
\paragraph{The circulant digraph $\overset{\bm{\rightarrow}}{\bm{C}}_{\bm{q}} \bm{{(i_1, \dots, i_k)}}$.} 
Consider the field $\mathbb{Z}_q$ and $i_1, \dots, i_k \in \mathbb{Z}_q$. A \emph{circulant digraph} $\overset{\rightarrow}{C}_q(i_1, \dots, i_k)$ has a vertex $v_i$ for each $i \in \{ 0, 1, \dots, q -1 \}$. There exists an arc $(v_{a},v_{b})$ if and only if $b \equiv  a + i \mod q$ for some $i \in i_1, \dots, i_k$. It is known that if $q = z(g-1) + 1$, then $\overset{\rightarrow}{C}_q( 1, 2, \dots, z)$ is a $[z,0;g]$-mixed graph (see e.g.~\cite{behzad1970minimal, araujo2009dicages}). 

The following conjecture is derived from the known upper bound on the order of circulant digraphs of the form $\overset{\rightarrow}{C}_q( 1, 2, \dots, z)$. 
\begin{conjec}[Behzad-Chartrand-Wall, \cite{behzad1970minimal}]\label{conj:order_of_digraph_is_verticesCirculant}
    Let $n[z,g]$ denote the minimum order of a $z$-regular digraph of girth $g$. Then, $n[z,g] = z(g-1) + 1$.
\end{conjec}

The conjecture has been shown true for $r = 2,3,4$ and for vertex-transitive digraphs (see e.g.~\cite{araujo2009dicages}).

\paragraph{The tree $\bm{\mathcal{T}_{r,g}}$ and lower Bounds for mixed cages.} Araujo-Pardo,  Hernández-Cruz, and Montellano-Ballesteros~\cite{araujo2019mixed} provided a very simple lower bound for the minimum order of $[1,r;g]$-mixed cages which is based on Moore's bound $n_0[r,g]$ that it is defined as:
    \begin{align*}
        n_0[r,g] = \begin{cases}
            \displaystyle 1 + r\sum_{i = 1}^{k-1} (r-1)^i & \text{ if the girth is $g = 2k+1$}, \\[1em]
            \displaystyle 2\sum_{i = 0}^{k-1} (r-1)^i & \text{ if the girth is $g = 2k$}. \\
        \end{cases}
    \end{align*}
Surprisingly, the bound for $n[1,r;g]$ turn out to be a very good lower bound as Exoo~\cite{exoo2023mixed} exhibited in his work, as mentioned in \cref{sec:intro}. 
This bound is obtained by counting the vertices of the mixed-tree $\mathcal{T}_{r,g}$, which has $(g+1)/2$ levels if $g$ is odd, and $g/2$ levels if $g$ is even.
$\mathcal{T}_{r,g}$ has a directed path $(v_1, \dots, v_g)$ at level 0 and it attaches a Moore tree of depth $i$ to vertices $v_{i+1}$ and $v_{g-i}$ for $i = 1, \dots, \lfloor \frac{g-1}{2} \rfloor$. That is, the tree has $g$ nodes at level 0, which we denote by $v_1, \dots, v_g$.
At level 1, for each node $v_i = v_2, \dots, v_{g-1}$, we add $r$ children $v_{i,1}, \dots, v_{i,r}$. 
At level 2, for each node $v_{i,j}$ with $i\in \{ 3, \dots, g-2\}$ and for all $j$, we add $r-1$ children $v_{i,j,1}, \dots, v_{i,j,r-1}$. 
At level 3, for each node $v_{i,j,k}$ with $i\in \{ 4, \dots, g-3\}$ and for all $j,k$, we add $r-1$ children $v_{i,j,k,1}, \dots, v_{i,j,k,r-1}$. 
And so on. If $g$ is odd, we also add extra $r-1$ children to each node of the level starting in $i=(g+1)/2$.
Finally, we direct the path  from $v_1$ to $v_g$ of $\mathcal{T}_{r,g}$. 

\begin{thm}[AHM lower bound,\cite{araujo2019mixed}] \label{thm:AHM_lower_bound}
     Let $n[1,r;g]$ be the order of a $[1,r;g]$-mixed cage. Then,
    \begin{equation} \displaystyle
        n[1,r;g] \geq n_{\text{AHM}}[1,r;g]  = \begin{cases}
            \displaystyle 2\left( 1 +\sum_{i = 1}^{k-1} n_0[r,2i+1] \right) + n_0[r,g] & \text{ if the girth is $g = 2k+1$}, \\[1.5em]
            \displaystyle 2\left( 1 +\sum_{i = 1}^{k-1} n_0[r,2i+1] \right) & \text{ if the girth is $g = 2k$}. \\
        \end{cases}
    \end{equation}
\end{thm}

\section{A family of mixed graphs with girth 6\label{sec:family_girth_6}}
In this section, we exhibit an infinity family of mixed graphs that have girth 6. We also show that this construction provides an upper bound for the minimum order of $[z,r;6]$-mixed cages, for certain pairs of $z$ and $r$. The main idea of our construction is to consider two copies of the undirected adjacency graph corresponding to the biaffine plane and connect them using arcs.

Let $q$ be a prime. Consider the bipartite graph $B_q$ over $\mathbb{Z}_q$ described in \cref{sec:preliminaries}. Take a copy of $B_q$ and refer to it as $B'_q$. For ease of notation, whenever $v'\in V(B'_q)$ appears, we assume that $v' = v$, that is, $v'$ is the copy vertex of the original vertex $v \in V(B_q)$; similarly for edges.

Let $p \in \mathbb{Z}$ be an odd number such that $2q = 4p +2$, or equivalently, $p = \frac{q-1}{2}$.

We describe the construction of digraph  $C_m$  for $m \in\mathbb{Z}_q$, which is a similar digraph to a circulant graph $\overset{\rightarrow}{C}_{2q}(1, 3, 5,  \dots, p)$. We use a different notation as the digraph we describe is also bipartite. $C_m$ has vertices $\{ [m,b] \, \vert \, b \in \mathbb{Z}_q \} \cup \{ [m',b'] \, \vert \, b \in \mathbb{Z}_q \}$. It has $p$ directed cycles, each one using the $2q$ vertices. A cycle consists on arcs between original and copy-vertices. Let us fix the the following ordering of the vertices by $[m,0], [m',0'], [m,1], [m',1'], \dots, [m,q-1], [m', (q-1)']$, which we call the \emph{alternating ordering}.
Then, we say that a cycle \emph{jumps} between vertices with a fixed length $j$ with respect to the alternating ordering if each of its arcs is of the form $(v_{i},v_{i+j \mod q})$ and either ($v_i \in V(B_q)$ and $v'_{i+j\mod q} \in V(B'_q)$) or ($v'_i \in V(B'_q)$ and $v_{i+j\mod q} \in V(B_q)$). Hence, for a cycle that has jumps of length $j \in \{ 1, 3, 5, \dots, p \}$, this cycle contains the arcs
\vspace{-2.5ex}
\begin{align*}
    \left( [m, b], [m',  (b+\frac{j-1}{2} )' ] \right), &&  //  \text{ \textit{From original to copy}} \\
    \left( [m', b'], [m, b + \frac{j+1}{2}] \right), && // \text{ \textit{From copy to original}}
\end{align*}
\vspace{-2.8ex} where the sums are taken in $\!\! \mod q$.
Note that this is equivalent to have arcs of type
\begin{align*}
    &\left( [m,b], [m', (b+j )' ] \right) \quad  j \in \{ 0,1,2, \dots, \frac{p-1}{2}\},&&  //  \text{ \textit{From original to copy}} \\  
    &\left( [m', b'], [m, b + j] \right) \quad j \in \{ 1,2,3, \dots, \frac{p+1}{2}\}, &&// \text{ \textit{From copy to original}}
\end{align*}
In the same fashion, we construct digraph $C_x$  for $x \in B_q$. \cref{fig:example_q_11} shows $C_m$ for $m \in \mathbb{Z}_{11}$.

\begin{figure}[t]
    \centering
    \begin{tikzpicture}[scale=.75,node distance=7em, 
    myArc/.style={draw,line width = 1.5pt, -{Stealth[length=4mm]}}, myEdge/.style={line width = 1.5pt}, 
    stateColor/.style={circle,  minimum size=2.2em, draw, line width = 1.2pt,orangeMedium},
    stateColorCopy/.style={circle,  minimum size=2.2em, draw, line width = 1.2pt,blueishMedium}]   
    
   \node[minimum size={2*4.25cm},regular polygon,regular polygon
    sides=22,rotate=11.75] at (3,3) (22-gon) {};
    \foreach \p [count=\n from 0] in {11,...,1}{%
        \tikzmath{integer \x, \y;
        \x = 2*\p; \y = 2*\p-1;}
        \node[stateColorCopy] at (22-gon.corner \y) (copy-\n) {$\n'$};
        \node[stateColor] at (22-gon.corner \x) (point-\n) {$\n$};};

    \foreach \i in {0, ..., 10}{
        \tikzmath{integer \j1, \j2, \j3;
        \j1 = Mod(\i+1,11); \j2 = Mod(\i+2,11); \j3 = Mod(\i+3,11); }
        \draw[myArc] (point-\i) to (copy-\i);
        \draw[myArc] (copy-\i) to (point-\j1);
        \draw[myArc, bend left=70] (point-\i) to (copy-\j1);
        \draw[myArc, bend left=90] (copy-\i) to (point-\j2);
        \draw[myArc, bend right=20] (point-\i) to (copy-\j2);
        \draw[myArc, bend right=30] (copy-\i) to (point-\j3);
        }
\end{tikzpicture}
    \caption{$C_m$ for $m \in \mathbb{Z}_{11}$, $p = 5$. Node $[m,b]$ is written simply by $b$. Original nodes are drawn in orange and copy nodes in blue. Outside arcs have jumps of length 3 while inside arcs have jumps of length 5. One shortest path is $(0,2',5,7',10,10')$.}
    \label{fig:example_q_11}
\end{figure}

The following result is not difficult to see by the fact that $C_m$ and $C_x$ are bipartite circulant graphs. Nonetheless, we show a proof for completeness. 

\begin{lem}\label{lem:digraph_Cm_is_mixed graph}
    $C_m$ and $C_x$ are $[\frac{p+1}{2}, 0 ; 6]$-mixed graphs for $m,x \in B_q$.
\end{lem}
\begin{proof}
    Without loss of generality we prove the statement for $C_m$.
    First, by construction each vertex clearly has $\frac{p+1}{2}$ in- and out-arcs, and $0$ edges.

    $C_m$ has a cycle of length 6. This cycle uses four arcs that correspond to jumps of length $p$ and then two arcs corresponding to jumps of length 1 (e.g., consider the cycle $\left( [m,0], [m', (\frac{p-1}{2})'], [m, p], [m', (\frac{3p-1}{2})'], [m, 2p], \right.\allowbreak \left. [m', (q-1)'] \right)$). 
    
    We claim that there is no small cycle than this type.  Suppose that $K$ is a cycle of length smaller than 6. 
    As $C_m$ is bipartite, then the length of $K$ cannot be 5 or 3. This implies that $|K| = 4$ as clearly $|K|=2$ is not possible by construction. Let $j_1, j_3, j_5, \dots, j_p$ be the length of the jumps made by $K$ w.r.t. the alternating ordering. Then, 
    \begin{equation}
        2q = a_1j_1 + a_3j_3 + \dots + a_pj_p \label{eq:cycle_K_structure}
    \end{equation}
    holds, where $a_j$ is the number of jumps of length $j$. By assumption on the size of $K$, even giving four jumps of the largest length is not enough to fulfill \cref{eq:cycle_K_structure} as $2q \neq 4p$. This shows that such cycle $K$ cannot exists, concluding the proof.
\end{proof}

Finally, consider the mixed graph $G_{p,q} = \left( V(B_q)\bigcup V(B'_q), E(B_q)\bigcup E(B'_q) \bigcup\limits_{m \in \mathbb{Z}_q} A(C_m) \bigcup\limits_{x \in \mathbb{Z}_q} A(C_x)\right)$.

\cref{fig:1-3-6-Construction} shows the construction with parameters $q=3$ and $p=1$, which determines a $[1,3;6]$-mixed graph.

\begin{thm}\label{thm:family_is_mixed_cage}
    For all $q\geq 3$ prime and $p = \frac{q-1}{2}$ such that $p$ is odd, the graph $G_{p,q}$ is a $[\frac{p+1}{2}, q; 6]$-mixed graph.
\end{thm}
\begin{proof}
    Parameters $\frac{p+1}{2}$ and $q$ are evident by construction and by \cref{lem:digraph_Cm_is_mixed graph}.
    For all $m,x \in \mathbb{Z}_q$, each digraph $C_m$ and $C_x$ has directed girth 6 by \cref{lem:digraph_Cm_is_mixed graph}. Similarly, the undirected graphs $B_q$ and $B'_q$ have undirected girth 6 (see \cref{sec:preliminaries}). We only need to verify that there is no mixed cycle of length smaller than 6, that is, that no new mixed cycle is formed by making the union of such graphs.

    Suppose that $K$ is a mixed cycle so that $|K|<6$. It is not difficult to see that $K$ cannot have length 2 nor 3.

    We show that $|K|\neq 4$. As we only have arcs between lines (points) that share the first same coordinate (one in vertex is in $B_q$ and the other in $B'_q$), in order to have a mixed cycle of length 4,  we need two edges: one in $B_q$ and one in $B'_q$. Thus, we require two arcs, one leaving $B_q$ (and entering $B'_q$) and one entering it  (and leaving $B'_q$). Consider a point $(x, mx+b)$ and its adjacent line (by edge) $[m,b]$. We know that $(x, mx+b)$ is arc-adjacent to $(x', (mx+b + j)')$ in $C_x$ with $j \in \{ 0,1, 2 \dots , \frac{p-1}{2} \}$. By properties of $B'_q$ (being the copy of $B_q$), $(x', (mx+b)')$ is edge-adjacent to $[m',b']$. By construction, $[m,b]$ is arc-adjacent to $[m',(b+j)']$, implying that there is no mixed cycle of length 4.

    We show that $|K|\neq 5$. To see this, note that the graph $G_{p,q}$ is a mixed bipartite graph with parts $\left\{[m,b], (x',y') \, \vert \, m, b \in V(B_q), \, x',y' \in V(B'_q) \right\}$ and $\left\{ [m',b'], (x,y) \, \vert \, m,b \in V(B'_q), \, x,y \in V(B_q) \right\}$. This implies that there is no mixed cycle of length 5.
\end{proof}

Note that $G_{p,q}$ has a lot of mixed cycles of length 6. To see this, consider the points $(x_1, mx_1 +b)$ and  $(x_2,mx_2 + b)$. Those points intersect on line $[m,b]$. Let $\left( x_1', (mx_1 +b + j)' \right)$ be the vertex arc-adjacent to 
$(x_1, mx_1 +b)$, for some $j \in \{ 1,2, 3, \dots, \frac{p+1}{2} \}$. Also, consider the vertex $\left( x_2', (mx_2 +b + j-1)' \right)$ which is arc-adjacent to $(x_2, mx_2 +b)$. Clearly, points $\left( x_1', (mx_1 +b + j)' \right)$ and $\left( x_2', (mx_2 +b + j-1)' \right)$ intersect in line $[m',b']$. This implies that we can find the mixed cycle of length 6 
\begin{equation*}
    \left( (x_1, mx_1 +b),  \left( x_1', (mx_1 +b + j)' \right),  [m',b'],  \left( x_2', (mx_2 +b + j-1)' \right),  (x_2,mx_2 + b),  [m,b] \right).
\end{equation*}

Note that the assumption on $p$ being odd can also be modified to allow even numbers. Hence, the jumps allowed on the circulant digraph are of the type $\{ 1,3,5, \dots, p-1 \}$.

\begin{cor}
    For all $q\geq 3$ prime and $p = \frac{q-1}{2}$ such that $p$ is even, the graph $G_{p,q}$ is a $[\frac{p}{2}, q; 6]$-mixed graph.
\end{cor}

Finally, the following corollary on the order of a mixed cage is trivial.
\begin{cor}
    Let $n[\frac{p+1}{2}, q; 6]$ be the minimum order of a $[\frac{p+1}{2}, q; 6]$-mixed cage with $q\geq 3$ a prime and $p = \frac{q-1}{2}$ odd. Then, $n[z,r;g] \leq 4q^2$.
\end{cor}

\cref{fig:1-3-6-Construction} shows that our construction for the $[1,3;6]$-mixed graph contains 36 nodes, while the $[1,3;6]$-mixed cage has 30 nodes (see \cref{sec:1-3-6_construction}).

\begin{figure}
    \centering
    \begin{tikzpicture}[scale=.9,
    myArc/.style={draw,line width = 1.5pt, -{Stealth[length=4mm]}}, 
    myEdge/.style={line width = 1.5pt}, 
    stateColor/.style={circle,  minimum size=2.6em, draw, line width = 1.2pt,orangeMedium},
    stateColorCopy/.style={circle,  minimum size=2.6em, draw, line width = 1.2pt,blueishMedium}]   
       \foreach \m in {0,1,2}{
            \foreach \n in {0,1,2}{
                \node[stateColor] at (0, -3.5*\m - 1.15*\n) (\m-point-\n) { };
                    \node at (\m-point-\n) (\m-point-\n-text) {\scriptsize $(\m,\n)$};
                \node[stateColor] at (0, -3.5*\m - 1.15*\n + 10.8) (\m-line-\n) { };
                     \node at (\m-line-\n) (\m-line-\n-text) {\scriptsize$[\m,\n]$};
                \node[stateColorCopy] at (3,-3.5*\m - 1.15*\n) (\m-pointCopy-\n) { };
                    \node at (\m-pointCopy-\n) (\m-pointCopy-\n-text) {\scriptsize$(\m',\n')$};
                \node[stateColorCopy] at (3, -3.5*\m - 1.15*\n + 10.8) (\m-lineCopy-\n) { };
                     \node at (\m-lineCopy-\n) (\m-lineCopy-\n-text) {\scriptsize$[\m',\n']$};};};
    
        \foreach \m in {0,1,2}{
            \foreach \b in {0,1,2}{
                \foreach \x in {0,1,2}{
                    \tikzmath{integer \y; \y = Mod(\m*\x + \b, 3);}
                    \draw[myEdge,orangeDark] (\m-line-\b) to[out= 180 -8*\x, in= 180 - 8*\m] (\x-point-\y);
                    \draw[myEdge, blueishDark] (\m-lineCopy-\b) to[out=8*\x, in= 8*\m] (\x-pointCopy-\y);};
                \tikzmath{integer \j1;
                \j1 = Mod(\b+1,3); }
                \draw[myArc] (\m-point-\b) to (\m-pointCopy-\b);
                \draw[myArc] (\m-pointCopy-\b) to (\m-point-\j1);
                \draw[myArc] (\m-line-\b) to (\m-lineCopy-\b);
                \draw[myArc] (\m-lineCopy-\b) to (\m-line-\j1);};};
    \end{tikzpicture}    
    \caption{Construction for the $[1,3;6]$-mixed graph with 36 nodes.}
    \label{fig:1-3-6-Construction}
\end{figure}

\section{Lower Bound for mixed cages\label{sec:lower_bounds}}

The family of mixed graphs of girth 6 inspired us for the following lower bound on the minimum order of $[z,r;g]$-mixed cages. The main idea is as follows. The circulant graph $\overset{\rightarrow}{C}_q(0, \dots, z)$ is one of the smallest directed graphs that are $z$-regular with girth $g$, where $q = z(g-1) + 1$. Clearly, adding any new arc (and thus edge) to this digraph creates a directed cycle. Consider $g$ consecutive nodes and add on them the tree $\mathcal{T}_{r,g}$ obtained from Moore's Lower bound as stated on \cref{sec:preliminaries}. This structure is clearly a lower bound for any $[z,r;g]$-mixed cage. See \cref{fig:lower-bound} for an example on with parameters $z=2, r=5,$ and $g=6$.

\begin{figure}
    \centering
    \begin{tikzpicture}[
    myArc/.style={draw,line width = 1.5pt, -{Stealth[length=4mm]}}, 
    myEdge/.style={line width = 1.5pt}, 
    state/.style={circle,  minimum size=2em, draw, line width = 1.2pt},
    stateTree/.style={circle,  minimum size=2em, draw, line width = 2pt, orangeMedium},
    stateSmall/.style={circle,  minimum size=1pt, draw, line width = 1.2pt}]    

        \node[stateTree] at (3.5, 27.3) (level0-0) {\normalsize 0};
        \node[stateTree] at (10, 27.3) (level0-1) {\normalsize  1};
        \node[stateTree] at (10, 23.3) (level0-2) {\normalsize 2};
        \node[stateTree] at (10, 14) (level0-3) {\normalsize 3};
        \node[stateTree] at (10, 10) (level0-4) {\normalsize 4};
        \node[stateTree] at (3.5, 10) (level0-5) {\normalsize 5};
        \foreach \j [count = \c from 0, count = \b from 1] in {10,...,6}{
            \coordinate (coord-\j) at ($(level0-\c)!0.5!(level0-\b)$);
            \newdimen\ey
            \pgfextracty{\ey}{\pgfpointanchor{coord-\j}{center}}   
            \node[state] at  (0,\ey) (level0-\j) {\normalsize \j};};
    
        \node[stateSmall, above = 1.8em of level0-1] (level1-1-1)  { };
        \node[stateSmall, below = 1.8em of level0-4] (level1-4-1)  { };
        \foreach \j \aux in {1/1,4/-1}{
            \node[stateSmall, right = 4em of level1-\j-1] (level1-\j-2)  { };
            \node[stateSmall, right = 4em of level1-\j-2] (level1-\j-3)  { }; 
            \node[stateSmall, left = 4em of level1-\j-1] (level1-\j-4)  { };
            \node[stateSmall, left = 4em of level1-\j-4] (level1-\j-5)  { };
            \foreach \i in {1,...,5}{
                \draw[myEdge] (level0-\j) to (level1-\j-\i);}
                \node[stateSmall, right = 4em of level1-\j-1] (level1-\j-2)  { };
                \node[stateSmall, right = 4em of level1-\j-2] (level1-\j-3)  { }; 
                \node[stateSmall, left = 4em of level1-\j-1] (level1-\j-4)  { };
                \node[stateSmall, left = 4em of level1-\j-4] (level1-\j-5)  { };
                \foreach \i in {1,...,5}{
                    \draw[myEdge] (level0-\j) to (level1-\j-\i);}
        }      
        
        \foreach \j in {2,3}{
            \node[stateSmall, right = 1.8em of level0-\j] (level1-\j-1)  { };
            \node[stateSmall, above = 4em of level1-\j-1] (level1-\j-2)  { };
            \node[stateSmall, above = 4em of level1-\j-2] (level1-\j-3)  { }; 
            \node[stateSmall, below = 4em of level1-\j-1] (level1-\j-4)  { };
            \node[stateSmall, below = 4em of level1-\j-4] (level1-\j-5)  { };
            \foreach \i in {1,...,5}{
                \draw[myEdge] (level0-\j) to (level1-\j-\i);};};
       
        \foreach \j in {2,3}{
            \foreach \i in {1,...,5}{
                \coordinate (coord2-\j-\i) at ($(level1-\j-\i) + (1,0)$);
                \node[stateSmall, above = 1pt of coord2-\j-\i] (level2-\j-\i-1)  { };
                \node[stateSmall, above = 1pt of level2-\j-\i-1] (level2-\j-\i-2)  { }; 
                \node[stateSmall, below = 1pt of coord2-\j-\i] (level2-\j-\i-3)  { };
                \node[stateSmall, below = 1pt of level2-\j-\i-3] (level2-\j-\i-4)  { };
                \foreach \k in {1,...,4}{
                \draw[myEdge] (level1-\j-\i) to (level2-\j-\i-\k);};};};
        \foreach \i in {0, ..., 10}{
            \tikzmath{integer \j1, \j2;
            \j1 = Mod(\i+1,11); \j2 = Mod(\i+2,11);}
            \draw[myArc] (level0-\i) to (level0-\j1);
            \draw[myArc, bend right] (level0-\i) to (level0-\j2);};
    \end{tikzpicture}
    \caption{Lower bound for a $[2,5;6]$-mixed cage with 11 + 66 - 6 = 69 nodes. Circulant graph has numbered vertices. The first level of vertices of the tree is shown in thick orange. }
    \label{fig:lower-bound}
\end{figure}

\begin{thm}\label{thm:lower_bound_mixedCages}
    Let $n[z,r;g]$ be the minimum order of a $[z,r;g]$-mixed cage. Then, $n[z,r;g] \geq n[z,g] + n_{\text{AHM}}[1,r;g] - g$. 
\end{thm}
\begin{proof}
    Any lower bound of $n[z,r;g]$ requires as many nodes as any $z$-regular directed graph with girth $g$. Over such digraph, we can take $g$ consecutive nodes and add $\mathcal{T}_{r,g}$. This provides $n_{\text{AHM}}[1,r;g] - g$ nodes, concluding the proof.
\end{proof}

Based on \cref{conj:order_of_digraph_is_verticesCirculant}, then we obtain the following.

\begin{conjec}\label{conj:lower_bound_mixedCages}
    If \cref{conj:order_of_digraph_is_verticesCirculant} is true, then $n[z,r;g] \geq z(g-1) + 1 + n_{\text{AHM}}[1,r;g] - g$.
\end{conjec}

Note that if $z=1$ the lower bound provided in \cref{thm:lower_bound_mixedCages} coincides with the AHM bound, but our construction contains an extra arc connecting the first and sixth vertices on the first level.

\section{The [1,3;6]-mixed cage from Galois' Projective Plane\label{sec:1-3-6_construction}}
In this section, we derive the $[1,3;6]$-mixed cage starting from the graph $G_{(2,4)}$ described in~\cref{sec:preliminaries}. Recall that $G_{(2,4)}$ represents the adjacency graph of the Galois projective plane $GF(2,4)$ and it contains 42 nodes. Exoo~\cite{exoo2023mixed} exhibited a $[1,3;6]$-mixed cage with 30 nodes that can be seeing as the connection of tree directed cycles with 10 vertices. The mixed cycles generated the girth equals to 6. It is not difficult to see that our construction is equivalent to that graph. 

\begin{thm}
    The $[1,3;6]$-mixed cage can be derived from $G_{(2,4)}$.
\end{thm}
\begin{proof}
    Recall that $G_{(2,4)}$ has girth 6.
    Consider the following construction,  where the main idea is to keep the girth while adding the missing arcs and trying to keep as many edges from  $G_{(2,4)}$ as possible. Clearly, we need to remove 12 vertices too. \cref{fig:cage3-1-6} shows the final graph.
    We refer as \emph{direct edge $\overrightarrow{uv\,\,}$} to the operation of deleting edge $uv$ and then adding arc $(u,v)$. Hence, the construction has the following steps.
    \begin{enumerate}
        \item Erase lines $[m,0], [m,1]$ for $m \in \{ 0,1 \}$, and $[\alpha,\alpha], [\alpha,\alpha^2]$. \\
        Erase points $(x,0), (x,1)$ for $x \in \{ 1, \alpha \}$, and $(\alpha^2,\alpha), (\alpha^2,\alpha^2) $. \label{item:step1} 
        \item Direct edges
        $\overrightarrow{[\alpha^2\alpha] P_{\alpha^2}}, \,\, \overrightarrow{P_{\alpha^2} L_\infty}, \,\,\overrightarrow{L_{\infty}P_\infty}, \,\, \overrightarrow{P_\infty L_0}, \,\, \overrightarrow{L_0 (0,\alpha^2)\,\,}$.\label{item:step2}
        \item Add arcs $(P_0, L_{\alpha^2})$, and $(P_1, L_\alpha)$.\\
        Add arc $(P_\alpha, L_{1})$ and orient edges $\overrightarrow{{[\alpha,0]} {(1,\alpha)}\,\,}$ and $\overrightarrow{[\alpha,1] (1,\alpha^2)\,\,}$.\label{item:step3}
        \item Add arc $\displaystyle \left(L_1, (0,\alpha)\right)$ and orient edges $\overrightarrow{(1,\alpha)[0,\alpha]\,\,}$ and $\overrightarrow{(1,\alpha^2)[1,\alpha]\,\,}$. \\
        Add arc $\displaystyle \left(L_\alpha, (0,1)\right)$ and orient edge $\overrightarrow{(\alpha,\alpha)[\alpha,1]\,\,}$.\\ 
        Add arc $\displaystyle \left(L_\alpha^2, (0,0)\right)$ and orient edge $\overrightarrow{(\alpha^2,1)[\alpha,0]\,\,}$.\label{item:step4}
        \item Add arcs $\displaystyle \left([\alpha^2,0],P_1\right) , \left( [\alpha^2,1],P_0 \right) , \left([\alpha^2,\alpha^2],P_\alpha \right)$. \label{item:step5}
        \item Add arc $\displaystyle \left( [0,\alpha], (\alpha^2,0)\right)$, orient edge $\overrightarrow{(0,\alpha)[\alpha^2, \alpha]\,\,}$ and remove edge $(\alpha^2,0) [1, \alpha^2]$.\label{item:step6} \\
        Add arc $\displaystyle \left([0,\alpha^2], (\alpha^2,1) \right)$, orient edge $\overrightarrow{(0,\alpha^2)[\alpha^2, \alpha^2]\,\,}$ and remove edge $(\alpha,\alpha^2) [0, \alpha^2]$.
        \item Add arcs $\displaystyle \left([1, \alpha],(\alpha,\alpha^2)\right)$ and $\displaystyle \left([1, \alpha^2],(\alpha,\alpha)\right)$.\label{item:step7}
        \item Add arcs $\displaystyle \left((0,1), [0, \alpha^2] \right)$, $\displaystyle \left((0,0), [1, \alpha^2] \right)$, $\displaystyle \left((\alpha^2,0),[ \alpha^2,0] \right)$, $\displaystyle \left((\alpha,\alpha^2),[\alpha^2,1] \right)$.\label{item:step8}
        \item Add edge $[\alpha^2,\alpha] (0,\alpha^2)$.\label{item:step9}
    \end{enumerate}

    \begin{figure}
         \centering
         \begin{tikzpicture}[scale=.75,node distance=7em, 
            myArc/.style={draw,line width = 1.5pt,blueishMedium, -{Stealth[length=4mm]}}, 
            myEdge/.style={line width = 1.5pt}, 
            state/.style={circle,  minimum size=2.9em, draw}, 
            stateColor/.style={circle,  minimum size=2.9em, draw,line width = 1.5pt},
            stateSub/.style={circle,  minimum size=2.3em, draw,line width = 1.5pt}]
            \node (P0) at (0,20) [state] {\scriptsize$P_0$};
            \node (P1) [state, below=10em of P0] {\scriptsize$P_1$};
            \node (P2) [state, below=10em of P1] {\scriptsize$P_\alpha$};
            \node (P3) [state, below =10em of P2] {\scriptsize$P_{\alpha^2}$};
            \node (L0) at ($(P3)+(11,0)$) [state] {\scriptsize$L_0$};
            \node (L1) [state, above=10em of L0] {\scriptsize$L_1$};
            \node (L2) [state, above=10em of L1] {\scriptsize$L_\alpha$};
            \node (L3) [state, above =10em of L2] {\scriptsize$L_{\alpha^2}$};
            \coordinate (neigh-P-0) at ($(P0)+(2.5,0)$);
            \node (l-0-1) [state, above = 0.23ex  of neigh-P-0,gray!40] {\scriptsize ${[0,\!1]}$};
            \node (l-0-0) [state, above  = 0.46ex of l-0-1,gray!40] {\scriptsize ${[0,\!0]}$};
            \node (l-0-2) [stateColor, below = 0.23ex of neigh-P-0, blueishLight] {\textcolor{black}{\scriptsize ${[0,\!\alpha]}$}};
            \node (l-0-3) [stateColor, below = 0.46ex of l-0-2,blueishDark] {\textcolor{black}{\scriptsize ${[0,\!\alpha^2]}$}};
            \coordinate (neigh-P-1) at ($(P1)+(2.5,0)$);
            \node (l-1-1) [state, above = 0.23ex  of neigh-P-1,gray!40] {\scriptsize ${[1\!,\!1]}$};
            \node (l-1-0) [state, above  = 0.46ex of l-1-1,gray!40] {\scriptsize ${[1\!,\!0]}$};
            \node (l-1-2) [stateColor, below = 0.23ex of neigh-P-1, orangeMedium] {\textcolor{black}{\scriptsize ${[1\!,\!\alpha]}$}};
            \node (l-1-3) [stateColor, below = 0.46ex of l-1-2,orangeDark] {\textcolor{black}{\scriptsize ${[1\!,\!\alpha^2]}$}};
            \coordinate (neigh-P-2) at ($(P2)+(2.5,0)$);
            \node (l-2-1) [stateColor, above = 0.23ex  of neigh-P-2, greenDark] {\textcolor{black}{\scriptsize ${[\alpha\!,\!1]}$}};
            \node (l-2-0) [stateColor, above  = 0.46ex of l-2-1, greenMedium] {\textcolor{black}{\scriptsize ${[\alpha\!,\!0]}$}};
            \node (l-2-2) [state, below = 0.23ex of neigh-P-2,gray!40 ] {\scriptsize ${[\alpha\!,\!\alpha]}$};
            \node (l-2-3) [state, below = 0.46ex of l-2-2,gray!40] {\scriptsize ${[\alpha\!,\!\alpha^2]}$};
            \coordinate (neigh-P-3) at ($(P3)+(2.5,0)$);
            \node (l-3-1) [stateColor, above = 0.23ex  of neigh-P-3,pinkMedium] {\textcolor{black}{\scriptsize ${[\alpha^2\!,\!1]}$}};
            \node (l-3-0) [stateColor, above  = 0.46ex of l-3-1, pinkLight] {\textcolor{black}{\scriptsize ${[\alpha^2\!,\!0]}$}};
            \node (l-3-2) [stateColor, below = 0.23ex of neigh-P-3, pinkDark] {\textcolor{black}{\scriptsize ${[\alpha^2\!,\!\alpha]}$}};
            \node (l-3-3) [stateColor, below = 0.46ex of l-3-2, pinkDark2] {\textcolor{black}{\scriptsize ${[\alpha^2\!,\!\alpha^2]}$}};
            \coordinate (neigh-L-0) at ($(L0)+(-2.5,0)$);
            \node (p-0-1) [stateColor, above = 0.23ex  of neigh-L-0, greenDark] {\textcolor{black}{\scriptsize ${(0,\!1)}$}};
                 \node [stateSub, pinkMedium] at (p-0-1) {};
            \node (p-0-0) [stateColor, above  = 0.46ex of p-0-1,greenMedium] {\textcolor{black}{\scriptsize ${(0,\!0)}$}};
                 \node [stateSub, pinkLight] at (p-0-0) {};
            \node (p-0-2) [stateColor, below = 0.23ex of neigh-L-0, blueishLight] {\textcolor{black}{\scriptsize ${(0,\!\alpha)}$}};
                 \node [stateSub, orangeMedium] at (p-0-2) {};
            \node (p-0-3) [stateColor, below = 0.46ex of p-0-2,blueishDark] {\textcolor{black}{\scriptsize ${(0,\!\alpha^2)}$}}; 
                 \node [stateSub, orangeDark] at (p-0-3) {};
            \coordinate (neigh-L-1) at ($(L1)+(-2.5,0)$);
            \node (p-1-1) [state, above = 0.23ex  of neigh-L-1,gray!40] {\scriptsize ${(1,1)}$};
            \node (p-1-0) [state, above  = 0.46ex of p-1-1,gray!40] {\scriptsize ${(1,0)}$};
            \node (p-1-2) [stateColor, below = 0.23ex of neigh-L-1, orangeDark] {\textcolor{black}{\scriptsize ${(1\!,\!\alpha)}$}};
                 \node [stateSub, pinkMedium] at (p-1-2) {};
            \node (p-1-3) [stateColor, below = 0.46ex of p-1-2,blueishDark] {\textcolor{black}{\scriptsize ${(1\!,\!\alpha^2)}$}};
                 \node [stateSub, pinkLight] at (p-1-3) {};
            \coordinate (neigh-L-2) at ($(L2)+(-2.5,0)$);
            \node (p-2-1) [state, above = 0.23ex  of neigh-L-2,gray!40] {\scriptsize ${(\alpha\!,\!1)}$};
            \node (p-2-0) [state, above  = 0.46ex of p-2-1,gray!40] {\scriptsize ${(\alpha\!,\!0)}$};
            \node (p-2-2) [stateColor, below = 0.23ex of neigh-L-2,blueishLight] {\textcolor{black}{\scriptsize ${(\alpha\!,\!\alpha)}$}};
                 \node [stateSub, pinkDark2] at (p-2-2) {};
            \node (p-2-3) [stateColor, below = 0.46ex of p-2-2, greenMedium] {\textcolor{black}{\scriptsize ${(\alpha\!,\!\alpha^2)}$}};
                 \node [stateSub, pinkDark] at (p-2-3) {};
            \coordinate (neigh-L-3) at ($(L3)+(-2.5,0)$);
            \node (p-3-1) [stateColor, above = 0.23ex  of neigh-L-3, orangeMedium] {\textcolor{black}{\scriptsize ${(\alpha^2\!,\!1)}$}};
                 \node [stateSub, pinkDark2] at (p-3-1) {};
            \node (p-3-0) [stateColor, above  = 0.46ex of p-3-1, greenDark] {\textcolor{black}{\scriptsize ${(\alpha^2\!,\!0)}$}};
                 \node [stateSub, pinkDark] at (p-3-0) {};
            \node (p-3-2) [state, below = 0.23ex of neigh-L-3,gray!40] {\scriptsize ${(\alpha^2\!,\!\alpha)}$};
            \node (p-3-3) [state, below = 0.46ex of p-3-2,gray!40] {\scriptsize ${(\alpha^2\!,\!\alpha^2)}$};
            \coordinate (C3) at ($(L1)!0.5!(L2)$);
            \node (Pi) at ($(C3)+(2,0)$) [state] {\scriptsize$P_\infty$};
            \coordinate (C4) at ($(P1)!0.5!(P2)$);
            \node (Li) at ($(C4)-(2,0)$) [state] {\scriptsize$L_\infty$};
            \foreach \u \v in {Li/P0,Li/P1,Li/P2, Pi/L1,Pi/L2,Pi/L3,
            P0/l-0-2,P0/l-0-3,
            P1/l-1-2,P1/l-1-3,
            P2/l-2-0, P2/l-2-1,
            P3/l-3-0, P3/l-3-1,P3/l-3-3,
            L0/p-0-0, L0/p-0-1,L0/p-0-2,
            L1/p-1-2,L1/p-1-3,
            L2/p-2-2,L2/p-2-3,
            L3/p-3-0, L3/p-3-1}{
                \draw[myEdge] (\u) to (\v);};
            \foreach \u \v in {P0/l-0-0, P0/l-0-1,P1/l-1-0, P1/l-1-1,P2/l-2-2,P2/l-2-3,
            L1/p-1-0, L1/p-1-1,L2/p-2-0, L2/p-2-1,L3/p-3-2,L3/p-3-3,
            l-1-3.east/p-3-0.north west,l-0-3/p-2-3%
            }{
                \draw[dotted, line width = 1pt, gray!40] (\u) to (\v);};
            \coordinate (C5) at ($(P0)!0.5!(L3)+(0,4.7)$);
            \draw[myArc]  (Li) to  [out=105,in=-180] (C5) to [out=0,in=75]  (Pi);
            \foreach \u \v in {l-3-2/P3, P3/Li,Pi/L0,L0/p-0-3,l-0-2/p-3-0, l-0-3/p-3-1,
            l-2-0/p-1-2, l-2-1/p-1-3, p-1-2/l-0-2.east,p-1-3/l-1-2,p-2-2/l-2-1,p-3-1/l-2-0.east, p-0-2/l-3-2,
           p-0-3/l-3-3%
            }{%
                \draw[myArc] (\u) to (\v);};
            \coordinate (C6) at ($(l-0-0)!0.5!(p-3-0)+(0,1.1)$);
            \draw[myArc]  (P0) to  [out=90,in=-180] (C6) to [out=0,in=90]  (L3);
            \coordinate (C7) at ($(l-1-0)!0.5!(p-2-0)+(0,1.1)$);
            \draw[myArc]  (P1) to  [out=90,in=-180] (C7) to [out=0,in=90]  (L2);
            \coordinate (C8) at ($(l-2-0)!0.5!(p-1-0)+(0,1.1)$);
            \draw[myArc]  (P2) to  [out=90,in=-180] (C8) to [out=0,in=90]  (L1);
            \foreach \u \v \bending in {L3.east/p-0-0.east/60,L2.south/p-0-1.east/30, L1/p-0-2/30, l-3-0.west/P1.south/40, l-3-1/P0/55, l-3-3.north west/P2.south/15}{
            \draw[myArc, bend left=\bending] (\u) to (\v);};
            \foreach \u \v in {l-1-2/p-2-3, l-1-3/p-2-2, p-0-1.west/l-0-3, p-0-0.west/l-1-3.east, p-3-0.south west/l-3-0.east, p-2-3/l-3-1.east}{
            \draw[myArc] (\u) to (\v);};
            \draw[myEdge, orangeLight] (l-3-2) to (p-0-3);
        \end{tikzpicture}
         \caption{[1,3;6]-mixed cage  from graph $G_{(2,4)}$.
         Erased edges and vertices are denoted by gray.
         Arcs are shown in blue.
         The new edge is shown in yellow.
         The rest of edges belong to the original graph $G_{(2,4)}$, where line $[m,b]$ colors points $(x_1,y_1)$ and $(x_2,y_2)$, and thus each point has two colors associated.}
         \label{fig:cage3-1-6}
     \end{figure}
    
    This graph has 30 vertices, and each one has three edges and one in- and one out-arc. We verify no cycle of length less than 6 is created in any of the steps. 

     Clearly, no new cycle is created in steps~\ref{item:step1} and ~\ref{item:step2}. 
     For step 3, orienting edge $\overrightarrow{{[\alpha,0]} {(1,\alpha)}}$ removes cycle $(P_\alpha, L_{1}, (1,\alpha), [\alpha,0])$, and orienting edge $\overrightarrow{[\alpha,1] (1,\alpha^2)}$ removes cycle $(P_\alpha, L_{1}, (1,\alpha^2)), [\alpha,1])$. Note that those mixed cycles have the new arc and three edges from $G_{(2,4)}$, hence removing or correctly directing an edge results in removing the mixed cycle; this is the main idea of the rest of the steps. This kind of cycle do not exist for pairs $P_0, L_{\alpha^2}$ and $P_1, L_\alpha$ as the mixed cycles of the form $([m,b], P_m, L_x, (x,y))$ for which $y=m x +b$ were already erased in  steps~\ref{item:step1} and ~\ref{item:step2}.     
     For steps~\ref{item:step4} and \ref{item:step5} it also holds that no new cycle is created, by a similar argument as in step~\ref{item:step3}.
     The same applies to steps~\ref{item:step6} and \ref{item:step7}, where in addition we can see that no new cycle is created by adding an arc between the set of points adjacent to line $L$ and the set of lines adjacent to point $P$; the arc added has the same \textit{direction} as arc $(P,L)$ added in step~\ref{item:step2}.
     Step~\ref{item:step8} can be seeing as a matching between lines missing out-arcs and points missing in-arcs, in such a way that no new cycle is added. Finally, only  $[\alpha^2,\alpha]$ and  $(0,\alpha^2)$ were missing one edge, which is added in step~\ref{item:step9}, and no new cycle is added by the arcs already defined between points and lines adjacent to $L_0$ and $P_{\alpha^2}$.
\end{proof}

Note that the graph $G_{(2,4)}$ is quite symmetrical. We were able to remove lines and points on pairs. For example, the deletion of lines $[0,0]$ and $[0,1]$, and points $(\alpha^2,\alpha)$  and $(\alpha^2,\alpha^2)$ allowed to add arc $(P_0,L_{\alpha^2})$ and no new cycle was created. Other projective planes lack of this symmetric structure, meaning that only removing certain lines and points is not enough to keep the girth equals to 6. That is, it is needed to remove or direct further edges, resulting on a complete different graph, if such graph exists.
Based on this, we conjecture that it is not possible to extend this construction using other finite fields $\mathbb{F}_q$ to generate $[1,r;6]$-mixed cages for $r\geq 4$, $q$ a prime power. 

\section*{Conclusions}
We provided a general lower bound for the minimum order of $[z,r;g]$-mixed cages for any value of $z,r$, and $g$.

We obtained an infinity family of mixed graphs of girth 6 using the biaffine plane. We also tried to generate families using the projective plane, but this seems a more challenging problem. Furthermore, we conjecture that this is not possible as the graphs generated by projective planes are already quite dense (those are already undirected cages of girth 6). In this sense, the Galois field of order 4 that generated the $[1,3;6]$-mixed cage is quite special. Its symmetrical structure allowed us to safely remove points and lines, and add the missing arcs.

There is still a long path to find families of girth 5 and 6 that provide better upper bounds.

\subsubsection*{Acknowledgments}
\ifnum\Anonimo=1 {
    We are thankful to György Kiss and to the Erasmus+20 scholarship to bring the authors together.

\paragraph{Funding:} G. Araujo-Pardo was supported by PAPIIT-UNAM-M{\' e}xico IN113324.

G. Araujo-Pardo and L.M. Mendoza-Cadena were supported CONAHCyT: CBF2023-2024-552 M{\' e}xico.

M. Mendoza-Cadena was supported by Centro de Modelamiento Matemático (CMM) BASAL fund FB210005 for center of excellence from ANID-Chile.
}
\fi

\paragraph{Disclosure statement:} The authors report there are no competing interests to declare.


\bibliographystyle{plain}
\bibliography{cagesGirth6.bib}
\end{document}